\documentclass[envcountsame,envcountsect]{svmult}

\usepackage{amsmath,amsfonts,amssymb,url}

\usepackage{mathptmx}       % selects Times Roman as basic font
\usepackage{helvet}         % selects Helvetica as sans-serif font
\usepackage{courier}        % selects Courier as typewriter font
\usepackage{type1cm}        % activate if the above 3 fonts are
\usepackage[utf8]{inputenc}

\newcommand{\eps}{\varepsilon}

\newcommand\Q{{\mathbb Q}}

\newcommand\Z{{\mathbb Z}}

\newcommand{\II}{{\mathcal I}}
\newcommand{\JJ}{{\mathcal J}}

\newcommand{\MM}{{\mathcal{M}}}
\newcommand{\NN}{{\mathcal{N}}}

\newcommand{\PP}{{\mathcal{P}}}

\newcommand{\bfg}{{\mathbf g}}

\newcommand{\pmin}{p_{\min}}
\newcommand{\pmax}{p_{\max}}

1

\makeatletter

\renewcommand*\l@section[2]{%
  \ifnum \c@tocdepth >\z@
    \addpenalty\@secpenalty
    \addvspace{0.2em \@plus\p@}%
    \setlength\@tempdima{1.5em}%
    \begingroup
      \parindent \z@ \rightskip \@pnumwidth
      \parfillskip -\@pnumwidth
      \leavevmode \bfseries
      \advance\leftskip\@tempdima
      \hskip -\leftskip
      #1\nobreak\hfil \nobreak\hb@xt@\@pnumwidth{\hss #2}\par
    \endgroup
  \fi}
  
\makeatother
 
\newcounter{jump}

\begin{document}

\hfuzz 3pt

\title*{Diversity in Parametric Families of Number Fields}
\author{Yuri Bilu, Florian Luca}
\institute{Yuri Bilu \at Institut de Mathématiques de Bordeaux, Université de Bordeaux \& CNRS; \email{yuri@math.u-bordeaux.fr}
\and 
Florian Luca \at School of Mathematics, Wits University, Johannesburg and Centro de Ciencias Matematicas, UNAM, Morelia; \email{Florian.Luca@wits.ac.za}}

\titlerunning{Diversity in Parametric Families of Number Fields}
\authorrunning{Yuri Bilu, Florian Luca}

\maketitle

\abstract*
{Let~$X$ be a projective curve defined over~$\Q$ and ${t\in \Q(X)}$ a non-constant rational function of degree ${\nu\ge 2}$.  For every ${ n \in \Z}$ pick ${P_ n \in X(\bar\Q)}$ such that ${t(P_ n )= n }$.  A result of Dvornicich and Zannier implies that, for  large~$N$,  among the number fields $\Q(P_1), \ldots, \Q(P_N)$ there are  at least $cN/\log N$  distinct; here ${c>0}$ depends only on the degree~$\nu$ and the genus ${\bfg=\bfg(X)}$.  We prove that there are at least ${N/(\log N)^{1-\eta}}$ distinct fields, where ${\eta>0}$  depends only on~$\nu$ and~$\bfg$.}

\abstract
{Let~$X$ be a projective curve defined over~$\Q$ and ${t\in \Q(X)}$ a non-constant rational function of degree ${\nu\ge 2}$.  For every ${ n \in \Z}$ pick a point ${P_ n \in X(\bar\Q)}$ such that ${t(P_ n )= n }$.  A result of Dvornicich and Zannier implies that, for  large~$N$,  among the number fields $\Q(P_1), \ldots, \Q(P_N)$ there are  at least $cN/\log N$  distinct; here ${c>0}$ depends only on the degree~$\nu$ and the genus ${\bfg=\bfg(X)}$.  We prove that there are at least ${N/(\log N)^{1-\eta}}$ distinct fields, where ${\eta>0}$  depends only on~$\nu$ and~$\bfg$.}

%%%%%%%%%%%%%%%%%%%%%%%%%%%%%%%%%%%%%%%%%%%%%
%{\footnotesize 

%\tableofcontents

%}

\section{Introduction}

\textit{Everywhere in this paper ``curve'' means ``smooth geometrically irreducible projective algebraic curve''. }

\bigskip

Let~$X$ be a  curve over~$\Q$ of genus~$\bfg$ and ${t\in \Q(X)}$ a non-constant rational function of degree ${\nu\ge 2}$. We fix, once and for all, an algebraic closure~$\bar\Q$. 
Our starting point is the celebrated Hilbert Irreducibility Theorem.

\begin{theorem}[Hilbert]
In the above set-up,
for infinitely many  ${ n \in \Z}$  the fiber ${t^{-1}( n )\subset X(\bar \Q)}$ is $\Q$-irreducible; that is, the Galois group $G_{\bar\Q/\Q}$ acts on $t^{-1}( n )$ transitively.
\end{theorem}  
This can also be re-phrased as follows: for every ${ n  \in \Z}$ pick ${P_ n \in t^{-1}( n )}$; then for infinitely many ${ n \in \Z}$ we have ${[\Q(P_ n ):\Q]=\nu}$.  %See Section~\ref{shilb} for a precise statement.

``Infinitely many'' in the Hilbert Irreducibility Theorem means, in fact, ``overwhelmingly many'': for sufficiently large positive~$N$ we have
\begin{equation}
\label{ehilb}
\bigl|\{n\in [1,N]\cap\Z: \text{$t^{-1}(n)$ is reducible}\}\bigr|\le c(\nu) N^{1/2}.
\end{equation}
%where the implies constant depends only on the degree~$\nu$. 
Everywhere in the introduction ``sufficiently large'' means ``exceeding a certain positive number depending on~$X$ and~$t$~''.

For the proof of~\eqref{ehilb} we invite the reader  to consult Chapter~9 of Serre's book~\cite{Se97}. See, in particular, Section~9.2 and the theorem on page 134 of~\cite{Se97}, where~\eqref{ehilb} is proved with~$\Q$ replaced by an arbitrary number field and~$\Z$ by its ring of integers. 

Hilbert's Irreducibility Theorem, however, does not answer the following natural question: among the field $\Q(P_ n )$, are there ``many'' distinct (in the fixed algebraic closure~$\bar\Q$)? This question is addressed in the article of Dvornicich and Zannier~\cite{DZ94}, where the following theorem is proved (see \cite[Theorem~2(a)]{DZ94}). 

\begin{theorem}[Dvornicich, Zannier]
\label{tdvz}
In the above set-up, there exists a real number ${c=c(\bfg,\nu)>0}$ %and ${N_0=N_0(X,t)> 1}$ 
such that for sufficiently large  integer~$N$ the number field ${\Q(P_1,\ldots, P_N)}$ is of degree at least $e^{cN/\log N}$ over~$\Q$.
\end{theorem}

One may note that the statement holds true independently of the choice of the points $P_ n $. 

An immediate consequence is the following result.

\begin{corollary}
\label{cdvz}
In the above set-up, there exists a real number ${c=c(\bfg,\nu)>0}$ such that for every sufficiently large integer~$N$, there are at least $cN/\log N$ distinct fields among the number fields $\Q(P_1), \ldots, \Q(P_N)$.  
\end{corollary}

Theorem~\ref{tdvz} is best possible, as obvious examples show.  Say, if~$X$ is (the projectivization of) the plane curve ${t=u^2}$ and~$t$ is  the coordinate function, then  the field 
$$
\Q(P_1,\ldots, P_N)=\Q(\sqrt1,\sqrt2, \ldots,\sqrt N)=\Q(\sqrt p: p\le N)
$$
is of degree ${2^{\pi(N)}\le e^{cN/\log N}}$. %The same holds true, by the Kummer theory, if the field extension $\bar\Q(X)/\bar\Q(t)$ is abelian and all its branch points 

On the contrary, Corollary~\ref{cdvz} does not seem to be best possible. For instance, in the same example, if~$ n $ runs the square-free numbers among ${1, \ldots, N}$ then the fields ${Q(P_ n )=\Q(\sqrt n )}$ are pairwise distinct. It is well-known that among ${1, \ldots, N}$ there are, asymptotically,  ${\zeta(2)^{-1}N}$ square-free numbers as ${N\to\infty}$.

We suggest the following conjecture.

\begin{conjecture}[Weak Diversity Conjecture]
\label{cours}
Let~$X$ be a  curve over~$\Q$ %of genus~$\bfg$ 
and ${t\in \Q(X)}$ a non-constant $\Q$-rational function of degree at least~$2$. Then there exists a real number ${c>0}$ such that  for every sufficiently large integer~$N$,  among the number fields $\Q(P_1), \ldots, \Q(P_N)$ there are at least $cN$ distinct. 
\end{conjecture}

There is also a stronger conjecture, attributed in~\cite{DZ94,DZ95} to Schinzel, which  relates to Theorem~\ref{tdvz} in the same way as Conjecture~\ref{cours} relates to Corollary~\ref{cdvz}. To state it, we need to recall the notion of  \textit{critical value}. 

We call ${\alpha \in \bar \Q\cup\{\infty\}}$ a \textit{critical value} (or a \textit{branch point}) of~$t$ if the rational function\footnote{Here and everywhere below we use the standard convention ${t-\infty=t^{-1}}$.} ${t-\alpha}$ has at least one multiple zero in $X(\bar \Q)$. It is well-known that any rational function ${t\in \bar \Q(X)}$ has at most finitely many critical values, and that~$t$ has at least~$2$ distinct critical values if it is of degree ${\nu\ge 2}$ (a consequence of the Riemann-Hurwitz formula). In particular, in this case~$t$ admits at least one \textit{finite} critical value.

\begin{conjecture}[Strong Diversity Conjecture (Schinzel)]
\label{cschin}
In the set-up of Conjecture~\ref{cours}, assume that either~$t$ has at least one  finite critical value not belonging to~$\Q$, or the field extension ${\bar \Q(X)/\bar\Q(t)}$ is not abelian. Then there exists a real number ${c>0}$ such that for every sufficiently large integer~$N$ the number field ${\Q(P_1,\ldots, P_N)}$ is of degree at least $e^{cN}$ over~$\Q$.
\end{conjecture}

As Dvornicich and Zannier remark, the hypothesis in the Strong Diversity Conjecture  is necessary. Indeed, when all critical values belong to~$\Q$ and the field extension ${\bar \Q(X)/\bar\Q(t})$ is  abelian, it follows from Kummer's Theory that $\Q(X)$ is contained in the field of the form ${L(t, (t-\alpha_1)^{1/e_1}, \ldots, (t-\alpha_s)^{1/e_s})}$, where~$L$ is a number field,  ${\alpha_1, \ldots, \alpha_s}$ are rational numbers and ${e_1, \ldots, e_s}$ are positive integers. Clearly, in this case the degree of the number field generated by ${P_1, \ldots, P_N}$ cannot exceed $e^{cN/\log N}$ for some ${c>0}$. 

On the other hand,    Conjecture~\ref{cours} does hold~\cite{BL16} in the case excluded in  Conjecture~\ref{cschin}, when the finite critical values of~$t$ are all in~$\Q$, and  the field extension ${\bar \Q(X)/\bar\Q(t)}$ is abelian.  Hence, \textit{the Strong Conjecture  implies the Weak Conjecture}. %\textbf{(to elaborate)}

Dvornicich and Zannier~\cite{DZ94,DZ95} obtain several results in favor of Schinzel's Conjecture. In particular, they show  that Conjecture~\ref{cschin} holds true in the following cases:

\begin{itemize}
\item
when $t$ admits a critical value of degree~$2$ or~$3$ over~$\Q$, see \cite[Theorem~2(b)]{DZ94}; 

\item
when all finite critical values are in~$\Q$ and the Galois group of the normal closure of $\bar\Q(X)$ over~$\bar\Q(t)$ is ``sufficiently large'' (for instance, symmetric or alternating), see~\cite{DZ95}. 

\end{itemize}

A result of Corvaja and Zannier \cite[Corollary~1]{CZ03} implies that, in the case when~$t$ has at least~$3$ zeros in $X(\bar \Q)$, a number field~$K$ of degree~$\nu$ or less may appear as $\Q(P_n)$ for at most $c(X,t,\nu)$ possible~$n$. In particular,  the Weak Conjecture holds in this case (but the Strong Conjecture remains open).

We mention also the work of Zannier~\cite{Za98}, who studies the following problem: given a number field~$K$, how many fields among ${\Q(P_1), \ldots, \Q(P_N)}$ contain~$K$? He proves that, under suitable assumptions, the number of such fields is $o(N^\eps)$ as ${N\to\infty}$  for any ${\eps>0}$.

%In the previous article~\cite{Bi16}  Theorem~\ref{tdvz} was extended from the base field~$\Q$ to an arbitrary number field. 

In the present article we go a different way: instead of  imposing  additional restrictions on~$X$ and~$t$, we work in full generality, improving on Corollary~\ref{cdvz} quantitatively in the direction of  Conjecture~\ref{cours}. Here is our principal result. 

%Our principal result is the following  refinement of Corollary~\ref{cdvz}. 

\begin{theorem}
\label{tmain}
In the set-up of Conjecture~\ref{cours}, there exists a positive real number ${\eta=\eta (\bfg,\nu)}$ %and  ${c=c(\bfg,\nu)>0}$ 
such that for every sufficiently large integer~$N$,  among the number fields $\Q(P_1), \ldots, \Q(P_N)$ there are at least ${N/(\log N)^{1-\eta}}$ distinct. 
\end{theorem}

The proof shows that ${\eta=10^{-6}\bigl((\bfg+\nu)\log(\bfg+\nu)\bigr)^{-1}}$ would do. 

\subsubsection*{Plan of the article}
In Section~\ref{snota} we introduce the notation and recall basic facts, to be used throughout the article. 

In Section~\ref{sargdz} we review the argument of Dvornicich and Zannier, and explain how it should be modified for our purposes. 

Sections~\ref{sset},~\ref{sgrege} and~\ref{spml} are the technical heart of the article. In Section~\ref{sset} and~\ref{sgrege} we introduce a  certain set of square-free numbers and study its properties. A key lemma used in Section~\ref{sgrege} is proved in Section~\ref{spml}. 

After all this preparatory work, the proof of Theorem~\ref{tmain} becomes quite transparent, see Section~\ref{sptmain}. 

\paragraph{Acknowledgments}
During the work on this article Yuri Bilu was partially supported by the University of Xiamen, and  by the binational research project MuDeRa, funded jointly by the French and the Austrian national science foundations ANR and FWF. 

We thank Jean Gillibert and Felipe Voloch  for useful discussions. We also thank the referees who carefully read the manuscript and detected several inaccuracies.

\section{Notation and Conventions}

\label{snota}

Unless the contrary is stated explicitly, everywhere in the article:
\begin{itemize}

\item
$n$ (with or without indexes) denotes a positive integer;

\item
$m$ (with or without indexes) denotes a square-free positive integer;

\item
$p$  (with or without indexes) denotes a prime number;

\item
$x$,~$y$,~$z$ denote positive real numbers.

\end{itemize}

We use the  notation 
$$
\pmax (n) =\max\{p: p\mid n\},\quad
\pmin (n) =\min\{p: p\mid n\}.
%\omega (n) =|\{p: p\mid n\}|.
$$
As usual, we denote by $\omega(n)$ (respectively $\Omega(n)$) the number of prime divisors of~$n$ counted without (respectively, with) multiplicities.

For a separable polynomial ${F(T)\in \Z[T]}$ we denote:

\begin{itemize}

\item
$\Delta_F$ the discriminant of~$F$;

\item
$\PP_F$ the set of~$p$ for which $F(T)$ has a root $\bmod\, p$, and which do not divide~$\Delta_F$.

\item

$\MM_F$ the set of square-free integers composed of primes from $\PP_F$.
\end{itemize}

By the Chebotarev Density Theorem, the set~$\PP_F$ is of positive density among all the primes. We call it the \textit{Chebotarev density} of~$F$ and denote it by~$\delta_F$. Note that 
\begin{equation}
\label{edelffac}
\delta_F\ge \frac1{d},  
\end{equation}
where ${d=\deg F}$.

%\section{Hilbert's Irreducibility Theorem}\label{shilb}

\section{The Argument of Dvornicich-Zannier}
\label{sargdz}

In this section we briefly review the beautiful ramification argument of Dvornicich and Zannier\footnote{In~\cite{DZ94} they trace it back to the work of Davenport et al~\cite{DLS64} from sixties.} and explain which changes are to be made therein to adapt it for proving Theorem~\ref{tmain}. 

%The letter~$p$ in the sequel denotes a rational prime number and~$ n $ denotes a positive integer. %For every such~$ n $ pick ${P_ n \in X(\bar \Q)}$ with ${t(P_ n )= n }$. 
Like in the introduction, in this section ``sufficiently large'' means ``exceeding some quantity depending on~$X$ and~$t$~''.

Let ${F(T)\in \Z[T]}$ be the primitive  separable  polynomial whose roots are exactly the finite critical values of~$t$, and let ${d=\deg F}$.  Using the Riemann-Hurwitz formula, one bounds the total number of critical values by ${2\bfg-2+2\nu}$, where  ${\bfg=\bfg(X)}$ is the genus of the curve~$X$. Hence 
\begin{equation}
\label{erihu}
d\le 2\bfg-2+2\nu.
\end{equation}
%It has the following properties. 
The basic properties of the polynomial $F(T)$ are summarized below. %in the following proposition. 

\renewcommand{\theenumi}{\textbf{\Alph{enumi}}}
\renewcommand{\labelenumi}{\theenumi}

%\begin{proposition}
\begin{enumerate}
\item
\label{iramdiv}
\textit{For sufficiently large~$p$, if~$p$ ramifies in $\Q(P)$ for some ${P\in t^{-1}( n )}$ then ${p\mid F( n )}$.}

\item
\label{idivram}
\textit{For sufficiently large~$p$, if ${p\,\|\,F( n )}$ then~$p$ ramifies in $\Q(P)$ for some ${P\in t^{-1}( n )}$. }

\item
\label{inuone}
\textit{For all~$p$  not dividing the discriminant $\Delta_F$ (which is  non-zero because~$F$ is a separable polynomial) the following holds: if for some~$ n $ we have ${p^2\mid F( n )}$ then ${p\,\|\,F( n +p)}$.}

\item
\label{idiv}
\textit{For every  ${p\in \PP_F}$  there exists  ${ n \le 2p}$ such that ${p\,\|\,F( n )}$.}

\item
\label{ifew}
\textit{When~$ n $ is sufficiently large, $F( n )$ has at most~$d$ prime divisors ${p\ge  n /4}$. }

\setcounter{jump}{\value{enumi}}

\end{enumerate}

%\end{proposition}

Here properties~\ref{iramdiv} and~\ref{idivram} are rather standard statements linking geometric and arithmetical ramification, see \cite[Theorem~7.8]{Bi16}. 

Property~\ref{inuone} is very easy: write 
$$
F(n+p)\equiv F(n)+F'(n)p\bmod p^2 . 
$$
If~$p^2$ divides both $F(n)$ and  ${F(n+p)}$ then ${p\mid F'(n)}$, which means that~$p$ must divide the discriminant $\Delta_F$, a contradiction. 

%Property~\ref{itcheb} is a consequence of the Chebotarev Density Theorem. 
Property~\ref{idiv} follows from~\ref{inuone}, and property~\ref{ifew} is obvious: if there are ${d+1}$ such primes, then ${(n/4)^{d+1}\le |F(n)|}$, which is impossible for large~$n$. 

One may also note that our definition of the polynomial $F(T)$ is relevant only for properties~\ref{iramdiv} and~\ref{idivram}; the other properties hold for any separable polynomial ${F(T)\in \Z[T]}$. 

\bigskip

Now we are ready to sketch the proof of Theorem~\ref{tdvz}. Denote by~$K_ n $ the number field ${\Q(t^{-1}( n ))}$, generated by all the points in the fiber of~$ n $, and by~$L_n$ the compositum of the fields ${K_1, \ldots, K_n}$. Then~$K_n$ is a Galois extension of~$\Q$ containing $\Q(P_n)$, and~$L_n$ is a Galois extension of~$\Q$ containing $\Q(P_1, \ldots, P_n)$.

We call~$p$  \textit{primitive} for some~$ n $ if~$p$ ramifies in~$K_ n $,  but not in~$L_{n-1}$. The observations above have the following two consequences.

\begin{enumerate}
\setcounter{enumi}{\value{jump}}

\item
\label{ipn}
\textit{Every sufficiently large ${p\in \PP_F}$ is primitive for some ${n\le 2p}$.}

\item
\label{inp}
\textit{Every sufficiently large~$n$ has at most~$d$ primitive ${p\in [n/4, n]}$. }

\end{enumerate}

Here~\ref{ipn} follows from~\ref{idivram} and~\ref{idiv}, and~\ref{inp} follows from~\ref{iramdiv}  and~\ref{ifew}.

For a given~$N$ let~$S_N$ be the set of~$n$ with the  property
$$
\text{$n$ has a primitive ${p\in [N/4, N/2]}$}. 
$$
It follows from~\ref{ipn} that ${S_N\subset [1,N]}$, and from~\ref{inp}, the Chebotarev Theorem and the Prime Number Theorem that, for sufficiently large~$N$
$$
|S_N|\ge \frac1d \bigl|\PP_F\cap[N/4,N/2]\bigr|\ge \frac{\delta_F}{5d} \frac N{\log N}.
$$
Furthermore, let $S'_N$ be the subset of~$S_N$ consisting of~$n$ such that the fiber $t^{-1}(n)$ is irreducible. The quantitative Hilbert Irreducibility Theorem~\ref{ehilb} implies that, for  large~$N$ we have  ${|S_N\smallsetminus S'_N|\le c(\nu)N^{1/2}}$, which means that, for large~$N$,
$$
|S'_N|\ge \frac{\delta_F}{6d}\frac N{\log N}.
$$

It is clear that if~$n$ admits a primitive~$p$ then~$K_n$ is not contained in $L_{n-1}$.   If, in addition to this, the fiber $t^{-1}(n)$ is irreducible, then $\Q(P_n)$ is not contained in ${\Q(P_1, \ldots, P_{n-1})}$, because in this case~$K_n$ is the Galois closure (over~$\Q$) of $\Q(P_n)$. It follows that
$$
[\Q(P_1, \ldots, P_N):\Q]\ge 2^{|S'_N|},
$$
which, in view of~\eqref{edelffac} and~\eqref{erihu}, proves Theorem~\ref{tdvz}.

\bigskip

The (already mentioned in the Introduction) example of the curve ${u=t^2}$ suggests that we can  make progress towards Conjecture~\ref{cours}  replacing prime numbers in the argument above by (suitably chosen) square-free numbers. This means that we have to obtain analogues of properties~\ref{ipn} and~\ref{inp} above with primes replaced by square-free numbers.

%First of all, we need to obtain 

Let~$m$ be a square-free integer, and~$n$ an arbitrary integer. We say that ${m\,\|\,n}$ if ${m\mid n}$ and ${\gcd(m,n/m)=1}$. %We denote by~$\omega(m)$ the number of distinct prime divisors of~$m$.

A ``square-free analogue'' of~\ref{ipn}  is relatively easy: one uses the following lemma, which generalizes property~\ref{inuone}.

\begin{lemma}
\label{lnuone}
Let~$m$ be a square free positive integer, coprime with $\Delta_F$ and such that  ${\pmin(m)>\omega(m)}$. Assume that for some~$n$ we have ${m\mid F(n)}$. Then there exists ${\ell\in \{0,1,\ldots, \omega(m)\}}$ such that ${m\,\|\,F(n+\ell m)}$. 
\end{lemma}

\begin{proof}
%The Chinese Remainder Theorem implies that for any ${m\in \MM_F}$ there exists 
Assume the contrary: for every ${\ell\in \{0,1,\ldots, \omega(m)\}}$ there exists ${p\mid m}$ such that ${p^2\mid f(n+\ell m)}$. By the box principle some~$p$ would occur for two distinct values~$\ell_1$ and~$\ell_2$; we will assume that ${0\le\ell_1<\ell_2\le \omega(m)}$. We obtain 
\begin{align*}
0&\equiv F(n+\ell_2 m)&&\bmod p^2\\
&\equiv F(n+\ell_1 m)+F'(n+\ell_1 m)(\ell_2-\ell_1)m&&\bmod p^2\\
&\equiv F'(n+\ell_1 m)(\ell_2-\ell_1)m &&\bmod p^2. 
\end{align*}
We have ${p\,\|\,m}$ and, since 
$$
0<\ell_2-\ell_1\le \omega(m)<\pmin(m)\le p,
$$ 
we have ${p\nmid (\ell_2-\ell_1)}$. Hence ${p\mid F'(n+\ell_1 m)}$, which implies that ${p\mid \Delta_F}$, a contradiction.\qed
\end{proof}

\bigskip

Recall that the set~$\PP_F$ consists of primes~$p$ not dividing the discriminant $\Delta_F$ and such that~$F$ has a root $\bmod\, p$, and that~$\MM_F$ is the set of square-free numbers composed of primes from~$\PP_F$. The following consequence is immediate. 

\begin{corollary}
\label{cdiv}
Let ${m\in \MM_F}$ have the property  ${\pmin(m)>\omega(m)}$. Then there exists  ${n\le m(\omega(m)+1)}$ such that ${m\,\|\,f(n)}$. 
\end{corollary}

\begin{proof}
The Chinese Remainder Theorem implies that for any ${m\in \MM_F}$ there exists ${n\le m}$ such that ${m\mid F(n)}$. Now use Lemma~\ref{lnuone}.\qed
\end{proof}

\bigskip

Call ${m\in \MM_F}$ \textit{primitive} for~$n$ if every ${p\mid m}$ ramifies in~$K_n$, and for every ${n'<n}$ some ${p\mid m}$ does not ramify in $K_{n'}$.  Combining Corollary~\ref{cdiv} with property~\ref{iramdiv}, we obtain a quite satisfactory generalization of property~\ref{ipn} to square-free numbers. 

{\sloppy

\begin{corollary}
\label{cprim}
Let~$m$ be like in Corollary~\ref{cdiv}. Then~$m$ is primitive for some ${n\le m(\omega(m)+1)}$. 
\end{corollary}

}

Another task to accomplish is  extending to square-free numbers property~\ref{inp}. This is much more intricate, see Sections~\ref{sset},~\ref{sgrege} and~\ref{spml}.

\section{A Special Set of Square-Free Numbers}
\label{sset}
In this section we fix a separable polynomial ${F(T)\in \Z[T]}$ of degree~$d$ and a real number~$\eps$ satisfying  ${0<\eps\le1/2}$. ``Sufficiently large'' will always mean ``exceeding a certain quantity depending on~$F$ and~$\eps$~'', and the constants implied by the ``~$O(\cdot)$~'' and ``~$\ll$~'' notation depend on~$F$ and~$\eps$ unless the contrary is stated explicitly.  

Recall that $\PP_F$ denotes the set of primes~$p$ not dividing the discriminant~$\Delta_F$ and such that~$F$ has a root $\bmod\, p$, and~$\MM_F$ denotes the set of the square-free numbers composed of primes from~$\PP_F$. Recall also that we denote by ${\delta=\delta_F}$ the density of~$\PP_F$. We have, as ${x\to\infty}$,
$$
\bigl|\PP_F\cap[0,x]\bigr|\sim \delta\frac x{\log x}, \qquad \bigl|\MM_F\cap[0,x]\bigr|\sim \gamma\frac x{(\log x)^{1-\delta}}
$$
where ${\gamma=\gamma(F)}$ is a certain positive real number. 

Recall that, unless the contrary is stated explicitly, the letter~$n$ always denotes a positive integer,~$m$ a square-free positive integer and~$p$ a prime number.

We fix a big positive real number~$x$ and set 
$$
\kappa=\log\log x, \qquad k=\lfloor\eps\delta\log\log x\rfloor+1, \qquad y=e^{(\log x)^{1-\eps}}. 
$$
Furthermore, we denote by $\MM_F(x)$ the set of ${m\in \MM_F}$  satisfying
$$
\frac x{2\kappa}\le m\le \frac x\kappa, \quad \pmax(m)\ge x^{9/10}, \quad \pmin(m)\ge y, \quad \omega(m)=k+1. 
$$

\begin{proposition}
\label{passy}
We have
${|\MM_F(x)|=x(\log x)^{-1+\varepsilon \delta+o(1)}}$
as ${x\to\infty}$. 
\end{proposition}

\begin{proof}

If ${m\in \MM_F(x)}$, then ${m=Pm_1}$, where ${P=\pmax(m)\ge x^{9/10}}$. % and ${m_1=m/P\le x^{1/10}}$. 
We denote by $\MM_F'(x)$ be the set of such~$m_1$'s. Then ${\MM'_F(x)\subset \MM_F}$ and for every
${m_1\in \MM_F'(x)}$  we have
\begin{equation}
\label{emone}
m_1\le x^{1/10}, \quad\pmin(m_1)\ge y, \quad \omega(m_1)=k.
\end{equation}

Let us count suitable~$P$ for a fixed $m_1$.  These are exactly the primes ${P\in \PP_F}$ from the interval ${[x/(2\kappa m_1), x/(\kappa m_1)]}$ satisfying ${P\ge x^{9/10}}$. The following observations are crucial.  
\begin{itemize}
\item
Since ${m_1\le x^{1/10}}$, we have ${x/(\kappa m_1)>x^{4/5}}$ for sufficiently large~$x$.
Hence, for a fixed $m_1$, 
the number of suitable~$P$ is bounded from above by 
$$
\pi\left(\frac x{\kappa m_1}\right)\ll \frac{x}{\kappa m_1 \log x}.
$$
\item
If ${m_1\le x^{1/10}/2\kappa}$ then every prime ${P\in \PP_F\cap[x/(2\kappa m_1), x/(\kappa m_1)]}$ is suitable. Hence,    for a fixed ${m_1\le x^{1/10}/2\kappa}$, 
the number of suitable~$P$ is bounded from below by 
$$
\pi_F\left(\frac x{\kappa m_1}\right)-\pi_F\left(\frac x{2\kappa m_1}\right)=\left(\frac\delta2+o(1)\right) \frac{x}{\kappa m_1 \log(x/(\kappa m_1))}\gg \frac{x}{\kappa m_1 \log x}.
$$
\end{itemize}
Here, $\pi_F(T)$ counts the number of primes in ${\PP_F\cap[0,T]}$. 

Summing up over ${m_1\in \MM_F'(x)}$, we obtain
\begin{equation}
\label{eboth}
\frac{x}{\kappa \log x}\sum_{\substack{m_1\in \MM_F'(x)\\  m_1\le x^{1/10}/2\kappa}} \frac{1}{m_1}\ll |\MM_F(x)|\ll\frac{x}{\kappa \log x}\sum_{m_1\in {\MM_F'(x)}} \frac{1}{m_1}.
\end{equation}
We will show that  the  the right-hand side of~\eqref{eboth} is bounded by ${x(\log x)^{-1+\varepsilon \delta+o(1)}}$ from above, and the left-hand side from below.  

\bigskip

The \textbf{upper bound} is easy:
\begin{align}
\sum_{m_1\in {\MM}_F'(x)} \frac{1}{m_1} & \le  \frac{1}{k!} \left(\sum_{\substack{y\le p\le x\\ p\in {\mathcal P}_F}} \frac{1}{p}\right)^k\nonumber\\
& \ll  \frac{1}{(k/e)^k} \left((\delta+o(1))\log\log x-(\delta+o(1))\log\log y\right)^k\nonumber\\
& \ll \left(\frac{(e+o(1))\eps\delta\log\log x}{k}\right)^k\nonumber\\
%& =  (e+o(1))^{\varepsilon \delta \log\log x}\nonumber\\
\label{eq:MFprime}
& =  (\log x)^{\varepsilon \delta+o(1)}
\end{align}
as ${x\to\infty}$. Hence,
${|\MM_F(x)|\le x(\log x)^{-1+\varepsilon \delta+o(1)}}$
as $x\to\infty$. 

\bigskip

For the \textbf{lower bound}, set 
${z=x^{(1/11\log\log x)}}$
and ${\II=[y, z]}$ and consider the following two sets:
\begin{itemize}
\item
the set $\MM_F''(x)$ of square-free numbers~$m_1$ with prime divisors in ${\PP_F\cap \II}$ and with ${\omega(m_1)=k}$;

\item
the set $\NN''_F(x)$ of \textit{non-square-free} numbers~$n_1$ with prime divisors in ${\PP_F\cap \II}$ and with ${\Omega(n_1)=k}$. 
\end{itemize}
Clearly, every ${m_1\in \MM''_F(x)}$ satisfies
$$
m_1\le x^{k/(11\log\log x)}<x^{1/11} \le \frac{x^{1/10}}{2\kappa}
$$
for large $x$. Hence the sum in the left-hand side of~\eqref{eboth} can be bounded as follows:
\begin{align}
\sum_{\substack{m_1\in \MM_F'(x)\\  m_1\le x^{1/10}/2\kappa}} \frac{1}{m_1}&\ge \sum_{m_1\in {\mathcal M}_F''(x)} \frac{1}{m_1}\nonumber\\
\label{edifference}
&\ge \frac{1}{k!} \left(\sum_{p\in {\mathcal P}_F\cap [y, z]} \frac{1}{p}\right)^k-\sum_{n_1\in \NN''_F(x)} \frac{1}{n_1}. %=S_1-S_2.
\end{align}
We need to esimate the first sum in~\eqref{edifference} from below and the second sum from above.

For the first sum we use the same argument as before and get
\begin{align*}
\frac{1}{k!} \left(\sum_{p\in {\mathcal P}_F\cap [y, z]} \frac{1}{p}\right)^k & \gg  \frac{1}{\sqrt{k}} \frac{1}{(k/e)^k} \left((\delta+o(1))\log \log z-(\delta+o(1))\log\log y\right)^k\\
& \gg   %\frac{1}{{\sqrt{\log\log x}}} 
\left(\frac{(e+o(1))\eps\delta\log\log x}{k}\right)^k\\
& =  (\log x)^{\varepsilon \delta+o(1)}.
\end{align*}
Now let us estimate the second sum in~\eqref{edifference}. Note that  every ${n_1\in \NN''_F(x)}$  satisfies ${n_1\le z^k<x}$ and is divisible by the square of a prime ${p\ge y}$. Hence, ${n_1=p^2 n_2}$ for some ${n_2\le x}$. It follows that
$$
\sum_{n_1\in \NN''_F(x)} \frac{1}{n_1}\le \left(\sum_{p\ge y} \frac{1}{p^2}\right)\left(\sum_{n_2\le x} \frac{1}{n_2}\right)\ll \frac{\log x}{y}=o(1)
$$
as ${x\to\infty}$. 

Putting all the estimates together, we conclude that 
$$
|\MM_F(x)|\gg \frac{x(\log x)^{\varepsilon \delta+o(1)}}{\log x\log\log x} = \frac{x}{(\log x)^{1-\varepsilon \delta+o(1)}}
$$
as $x\to\infty$, which is what we wanted. \qed
\end{proof}

\section{Greedy and Generous Square-free Numbers}
\label{sgrege}

\textit{We retain the notation and set-up of  Section~\ref{sset}. }

\bigskip

As we have already remarked in Section~\ref{sargdz}, the Chinese Remainder Theorem implies that for any ${m\in \MM_F}$ there exists a positive integer~$n$ such that ${m\mid F(n)}$. Moreover, if ${m\in \MM_F(x)}$ then we can choose such~$n$ satisfying ${n\le x}$.  Of course, there can be several~$n$ with this property; pick one of them and call it~$n_m$. 

Thus,  for every ${m\in \MM_F(x)}$ we pick ${n_m\le x}$ such that ${m\mid f(n_m)}$; we fix this choice of the numbers~$n_m$ until the end of this section. 

It might happen that ${n_m=n_{m'}}$ for distinct ${m,m'\in \MM_F(x)}$. It turns out, however, that, with a suitable choice of our parameter~$\eps$, %, when~$m$ runs through the set $\MM_F(x)$ and we select ${n_m\le x}$, 
the repetitions are ``not too frequent''. 

Call ${m\in \MM_F(x)}$ \textit{generous} if it shares its~$n_m$ with at least $6d$ other elements of $\MM_F(x)$, and \textit{greedy} otherwise.

%Precisely, we have the following statement.

\begin{proposition}
\label{psect}
%Assume that for every ${m\in \MM_F(x)}$ we have ${n_m\le x}$ and 
Specify
\begin{equation}
\label{eeps}
\eps=\frac1{10^3\log(2d)}. 
\end{equation}
Then for sufficiently large~$x$ at least half of the elements of the set $\MM_F(x)$ are greedy. In particular, 
$$
\bigl|\{n_m: m\in \MM_F(x)\}\bigr|\ge \frac1{12d}|\MM_F(x)|.
$$
\end{proposition}

The crucial tool in the proof of this proposition is the following lemma, which might be viewed as a partial ``square-free'' version of Property~\ref{ifew} from Section~\ref{sargdz}. We cannot affirm that $F(n)$ has ``few'' divisors in $\MM_F$ for all~$n$; but we can affirm that, with ``few'' exceptions, $F(n)$ has ``few'' divisors in $\MM_F(x)$. 

\begin{lemma}
\label{lthree}
For sufficiently large~$x$, the set of ${n\le x}$ such that $F(n)$ has more than $6d$ divisors in $\MM_F(x)$, is of cardinality at most ${x(\log x)^{-2+30\eps\log(2d)}}$. 
\end{lemma}

We postpone the proof of this lemma until Section~\ref{spml}. 

\subsection{Initializing the Proof of Proposition~\ref{psect}}
Starting from this subsection we work on the proof of Proposition~\ref{psect}.

We set ${\JJ=[y,x]}$ and we try to understand the function $\omega_\JJ(F(n))$, where $\omega_\JJ(\cdot)$ is the number of prime factors of the argument in the interval~$\JJ$. We split~$n$ into three sets as follows.
\renewcommand{\theenumi}{\roman{enumi}}
\renewcommand{\labelenumi}{(\theenumi)}
\begin{enumerate}
\item 
\label{ienorm}
$E(x)$ (enormous), which is the set of ${n\le x}$ for which 
$$ 
\omega_\JJ(F(n))\ge 3d (\log\log x)^2.
$$
\item
\label{ilarge} 
$L(x)$ (large), which is the set of ${n\le x}$ for which
$$
\omega_\JJ(F(n))\in [10^5d^2\log\log x, 3d(\log\log x)^2].
$$
%where ${K=10^5 d}$.% is some constant depending on~$d$ to be specified later.
\item
\label{ireas} 
$R(x)$ (reasonable), which is the set of $n\le x$ such that 
$$
\omega_\JJ(F(n))\le 10^5d^2\log\log x.
$$
\end{enumerate}

For the purpose of this argument, if ${s=\omega_\JJ(F(n))}$ then we  denote all the prime factors of $F(n)$ in~$\JJ$ by
${p_1<p_2<\cdots<p_s}$
.

We will use the multiplicative function $\rho_F$, defined for a positive integer~$u$ by 
\begin{equation}
\label{erhof}
\rho_F(u)=|\{ 0\le n\le u-1: F(n)\equiv0\bmod u\}|.
\end{equation}
Clearly, $\rho_F(m)\le d^{\omega(m)}$ holds for all squarefree positive integers $m$.  

\subsection{Counting~$m$ with  ${n_m\in E(x)}$}

Since 
${|F(n)|\ll n^d\ll x^d}$
it follows that in case~(\ref{ienorm}), if we put ${U=\lfloor (\log\log x)^2\rfloor}$, then ${p_1\cdots p_U\le x^{1/2}}$ for large $x$.

To count $E(x)$, fix ${p_1<p_2<\cdots<p_U}$ all in~$\JJ$ and let us count the number of ${n\le x}$ such that 
${m_1\mid f(n)}$, where ${m_1=p_1\cdots p_U}$. The number of such~$n$ is 
\begin{equation}
\label{eq:2}
\frac{\rho_F(m_1)}{m_1} x+O(\rho_F(m_1))\ll \frac{d^{\omega(m_1)}}{m_1} x+d^{\omega(m_1)}\ll \frac{d^{\omega(m_1)}}{m_1}x.
\end{equation}
%For the above inequality, we used that
%$$d^{\omega(m_1)}=d^{O(\log m_1/\log\log m_1)}=m_1^{o(1)}\le x^{o(1)},$$
%so in 
In the middle of~\eqref{eq:2}, the first term ${d^{\omega(m_1)}x/m_1}$ dominates because ${m_1\le x^{1/2}}$. 

We sum up over the possible $m_1$ getting
\begin{equation}
\label{eq:3}
|E(x)|\ll xd^U\sum%_{\substack{p\mid m_1\Rightarrow p\in [y,x]\\ \mu^2(m_1)=1\\ \omega(m_1)=U}} 
\frac{1}{m_1},
\end{equation}
where the sum runs over all square-free~$m_1$ satisfying ${\omega(m_1)=U}$ and having all prime divisors in~$\JJ$. 
We estimate this sum by the multinomial coefficient trick, already used in the proof of Proposition~\ref{passy}: 
%The last sum which we denote by $S_3$, is, by the , 
\begin{align*}
\sum\frac1{m_1}  \ll  \frac{1}{U!} \left(\sum_{y\le p\le x} \frac{1}{p}\right)^U 
 \ll  \left(\frac{3\log\log x}{U}\right)^{U}
\end{align*}
This gives us the estimate 
\begin{equation*}
|E(x)|\ll  x\left(\frac{3d\log\log x}{U}\right)^{U},
\end{equation*}
which, with our definition ${U=\lfloor (\log\log x)^2\rfloor}$, implies that 
$$
|E(x)|\le xe^{-(1+o(1))(\log\log x)^2 \log \log\log x}
$$
as ${x\to \infty}$.

Having bounded  $|E(x)|$, we may now estimate the number of~$m$ such that ${n_m\in E(x)}$. 
For each  ${n\le x}$ we have ${|F(n)|\ll n^d\le x^d}$ which implies that, for large~$x$, we have ${\omega_\JJ(F(n))\le \log x}$. Thus, for large~$x$, the divisor ${m\mid F(n)}$ with ${\omega(m)=k}$ can be chosen in at most
$$
\binom{\lfloor \log x\rfloor}{k+1}\le (\log x)^{k+1}\ll e^{2(\log\log x)^2}
$$
ways. 
This implies that, as ${x\to\infty}$, 
\begin{align*}
\bigl|\{m\in \MM_F(x): n_m\in E(x)\}\bigr|&\le |E(x)|e^{2(\log\log x)^2}\\
&\le xe^{-(1+o(1))(\log\log x)^2 \log \log\log x}.
\end{align*}
Proposition~\ref{passy} implies that this is $o\bigl(|\MM_F(x)|\bigr)$ as ${x\to\infty}$. %So, we dealt with (i).

\subsection{Counting~$m$ with ${n_m\in L(x)}$}

Let us deal with~(\ref{ilarge}) now. We let~$i_0$ and~$i_1$ be the maximal and the minimal positive integers such that ${2^{i_0}\le 10^5d}$ and ${2^{i_1}\ge 3(\log\log x)}$, respectively. Clearly, ${i_1-i_0=O(\log\log\log x)}$. Consider an integer
${j\in [i_0,i_1-1]}$ and denote by $L_j(x)$ the subset of $L(x)$ consisting of~$n$ such that 
$$
\omega_\JJ(F(n))\in [2^j d\log\log x, 2^{j+1}d\log\log x].
$$

We revisit the previous argument. We now take $U=\lfloor 2^{j-1} \log\log x\rfloor$, and let 
${m_1=p_1\cdots p_U}$.
Then $m_1^{2d}\le |F(n)|\ll x^{d}$, therefore $m_1\ll x^{1/2}$. Now exactly as before we prove that 
\begin{equation*}
|L_j(x)|\ll  x\left(\frac{3d\log\log x}{U}\right)^{U},
\end{equation*}
which, with our definition ${U=\lfloor 2^{j-1}\log\log x\rfloor}$, implies that 
$$
|L_j(x)|\ll \frac x{(\log x)^{2^{j-2} \log(2^{j-2}/3d)}}.
$$
Since
$$
\log\frac{2^{j-2}}{3d}\ge \log\frac{2^{i_0-2}}{3d}\ge \log \frac {10^5d}{24d}\ge 8, 
$$
we have
$$
|L_j(x)|\ll \frac x{(\log x)^{2^{j+1}}}.
$$

On the other hand, 
for ${n\in L_j(x)}$ we have ${\omega_\JJ(F(n))\le 2^{j+1}d\log\log x}$. It follows that, for large~$x$,   the number of choices for~$m$ for a given ${n\in L_j(x)}$ is at most
\begin{align}
\binom{\lfloor2^{j+1}d\log\log x\rfloor}{k+1}&\le \frac{\bigl(2^{j+1}d\log\log x\bigr)^{k+1}}{(k+1)!}\nonumber\\
&\le \left(\frac{2^{j+3}d}{\delta\eps}\right)^{2\delta\eps\log\log x}\nonumber\\
\label{eexp}
&=(\log x)^{2\delta\eps\log \left(2^{j+3}d/\delta\eps\right)}.
\end{align}
Since %${K= 10^5d}$, we have 
$$
\frac{2^{j-1}}{\delta\eps}\ge 2^{i_0-1}\ge \frac {10^5d}4 \ge 10^4d,
$$
we have
$$
\frac{2^{j-1}}{\delta\eps}\ge 2\log\frac{2^{j-1}}{\delta\eps}\ge \log\left(\frac{2^{j-1}}{\delta\eps}\cdot 10^4d\right)\ge 
\log\frac{2^{j+3}d}{\delta\eps},
$$
which shows that the exponent in~\eqref{eexp} does not exceed $2^j$. 

Thus, for large~$x$ 
$$
\bigl|\{m\in \MM_F(x): n_m\in L_j(x)\}\bigr|\le |L_j(x)|(\log x)^{2^j} \ll \frac x{(\log x)^{2^j}} \le \frac x{(\log x)^2},
$$
because ${2^j\ge 2^{i_0}\ge 10^5d/2\ge 2}$. 
Since there are $O(\log\log\log x)$ possible~$j$, we conclude that 
$$
\bigl|\{m\in \MM_F(x): n_m\in L(x)\} \ll  \frac {x\log\log\log x}{(\log x)^2},
$$
which is again $o\bigl(|\MM_F(x)|\bigr)$ as ${x\to\infty}$.

Thus, we have proved that 
\begin{equation}
\label{eosmall}
\bigl|\{m: n_m\in E(x)\cup L(x)\}\bigr|=o\bigl(|\MM_F(x)|\bigr)
\end{equation}
as ${x\to\infty}$.

\subsection{Completing the proof}

We are ready now to complete the proof of Proposition~\ref{psect}. 
It remains to deal with $n\in R(x)$. If $n\in R(x)$, then ${\omega_\JJ(F(n))\le 10^5d^2\log\log x}$. Thus, for  fixed ${n\in R(x)}$ we have 
%the number of $m$'s such that $(n,m)$ is in ${\mathcal L}$  is 
\begin{align}
\bigl|\{m\in \MM_F(x): n_m=n\}\bigr|&\le \binom{\lfloor 10^5d^2 \log\log x\rfloor }{k+1}\nonumber\\
&\le \frac{(10^5d^2 \log\log x)^{k+1}}{(k+1)!}\nonumber\\
&\le \left(\frac{10^6d^2}{\eps\delta}\right)^{2\eps\delta \log\log x}\nonumber\\
\label{ehowmany}
&= (\log x)^{2\eps\delta\log(10^6d^2/\eps\delta)}.
\end{align}

Now we are done:  Lemma~\ref{lthree} combined with  estimate~\eqref{ehowmany} implies that  there exists at most 
\begin{equation}
\label{elong}
\frac x{(\log x)^{2-30\eps\log(2d)-2\eps\delta\log(10^6d^2/\eps\delta)}}
\end{equation} 
generous ${m\in \MM_F(x)}$ with the property ${n_m\in R(x)}$. When~$\eps$ is chosen as in~\eqref{eeps}, a quick calculation shows that 
$$
30\eps\log(2d)+2\eps\delta\log\left(\frac{10^6d^2}{\eps\delta}\right)<\frac12. 
$$
Hence~\eqref{elong} is $o(|\MM_F(x)|)$ as ${x\to \infty}$. In particular,  when~$x$ is sufficiently large, at least half of elements of $\MM_F(x)$ are greedy.
 \qed

\bigskip

It remains to prove Lemma \ref{lthree}.

\section{Proof of Lemma~\ref{lthree}}
\label{spml}

\textit{We keep the notation of Section~\ref{sset}, especially ${y=\exp((\log x)^{1-\varepsilon})}$. 
}

\subsection{Two Simple Lemmas}

\label{sstwolem}

Let~$A$ be the subset of~$\MM_F$ consisting of~$m$ with ${\pmin(m)\ge y}$. 
We study the set ${A(z)=A\cap[y,z]}$ for ${z\in [y,x]}$. 

\begin{lemma}
\label{lfive}
When~$x$ is sufficiently large we have ${|A(z)|\le z(\log x)^{-1+3\eps}}$ for all ${z\in [y,x]}$. 
\end{lemma} 

\begin{proof}
Let $g(n)$ be the characteristic function of $A$. Then for any ${z>1}$ we have 
$$
\sum_{p\le z}g(p)\log p \le 2z,
$$
and ${g(p^n)=0}$ for ${n\ge2}$. 
Using Lemma~9.6 on page 138 in \cite{DKL12}, we obtain
\begin{equation}
\label{esix}
|A(z)|=\sum_{n\le z} g(n)\le3 \frac z{\log z} \sum_{n\in A(z)} \frac{1}{n}.
\end{equation}
Clearly, ${\log z\ge(\log x)^{1-\eps}}$ for ${z\in [y,x]}$. As for the sum above, we have
$$
\sum_{n\in A(z)} \frac{1}{n}\le   \prod_{y\le p\le z} \left(1+\frac{1}{p}\right) \le  (\log x)^{\eps+o(1)}
$$
as ${x\to\infty}$. Together with \eqref{esix} this finishes the proof.  \qed
\end{proof}

\begin{lemma}
\label{lsix}
Assuming~$x$ sufficiently large, for  ${y\le a\le b\le x}$ we have 
$$
\sum_{\substack{a\le n\le b\\ n\in A}} \frac{1}{n}\le \frac{\log b-\log a+1}{(\log x)^{1-3\eps}}.
$$
\end{lemma}

\begin{proof}
Using Abel summation and Lemma~\ref{lfive}, we obtain
\begin{align*}
\sum_{\substack{a\le n\le b\\ n\in A}} \frac{1}{n}&=\int_a^b\frac{d|A(z)|}z\\
&=\frac{|A(b)|}b-\frac{|A(a)|}a+\int_{a}^b \frac{|A(z)|}{z^2} dz\\
&\le \frac{|A(b)|}b+\frac1{(\log x)^{1-3\eps}}\int_{a}^b \frac{dz}{z}\\
&\le \frac1{(\log x)^{1-3\eps}}+\frac{\log b-\log a}{(\log x)^{1-3\eps}},
\end{align*}
as wanted. \qed

\end{proof}

\subsection{Cliques}
Starting from this subsection we begin the proof of Lemma \ref{lthree}. 
Recall that every ${m\in \MM_F(x)}$ writes as ${m=m_1P}$, where ${P=\pmax(m)\ge x^{9/10}}$. As in Section~\ref{sset} we denote by $\MM'_F(x)$ the set of all~$m_1$ obtained this way. They satisfy~\eqref{emone}, which will be used in the sequel without special reference.

Let ${n\le x}$ be such that $F(n)$ has at least $6d$ distinct divisors in $\MM_F(x)$. Write each of them  ${m_1P}$ as above and let~$s$ be the number of such~$P$. Then
${x^{9s/10} \le |f(n)|\ll x^{d}}$,
so ${s\le 10d/9+o(1)}$ as ${x\to\infty}$. In particular, ${s<2d}$ for large~$x$. Hence among the $6d$  divisors there are three with the same~$P$; write them ${m_1P}$, ${m_2P}$ and $m_3P$.  

Let us call an (unordered) triple of pairwise distinct ${m_1,m_2,m_3\in \MM_F'(x)}$ a \textit{clique} if there exists a prime ${P\ge x^{9/10}}$ such that ${m_1P,m_2P,m_3P\in \MM_F(x)}$. If ${\{m_1,m_2,m_3\}}$ is a clique then ${m_1P,m_2P,m_3P\in [x/(2\kappa), x/\kappa]}$. This implies that in a clique we have
\begin{equation} 
\label{ecli}
\frac{m_j}2\le m_i\le 2m_j
\end{equation} 
for any $i,j$. In addition to this, 
since ${m_1,m_2,m_3}$ in a clique are square-free with the same number of prime factors, we have
\begin{equation}
\label{enot}
\gcd(m_i,m_j)<m_i<[m_i,m_j],
\qquad (i\ne j).
\end{equation} 
where ${[\cdots]}$ denotes the least common multiple. We will repeatedly use these properties. 

\subsection{The Sum over Cliques}

To prove the lemma, it suffices to estimate the number of~$n$ such that $F(n)$ has three distinct divisors forming a clique. When a clique ${\{m_1,m_2,m_3\}}$ is fixed, the  number of such~$n$ is at most
\begin{equation}
\label{enumnfim}
\frac{\rho_F([m_1,m_2,m_3])}{[m_1,m_2,m_3]} x+O(\rho_F([m_1,m_2,m_3])),
\end{equation}
where $\rho_F(\cdot)$ is defined in~\eqref{erhof}.   When~$x$ is large, we have
$$
\omega([m_1,m_2,m_3])\le 3k\le 4\varepsilon \log\log x, 
$$
which implies
$$
\rho_F([m_1,m_2,m_3])\le d^{\omega([m_1,m_2m_3]}\le(\log x)^{4\eps\log d}.
$$
Further, since ${m_i\le x^{1/10}}$, we have ${[m_1m_2,m_3]\le x^{3/10}\le x^{1/2}}$. It follows that in~\eqref{enumnfim} the first term dominates over the second one, and 
the number of our~$n$ (for the fixed ${m_1,m_2,m_3}$) is bounded, for large~$x$, by
$$
x(\log x)^{5\eps\log d} \frac{1}{[m_1,m_2,m_3]}.
$$

Hence the total number of~$n$ (for all possible choices of ${m_1,m_2,m_3}$) is bounded by ${x(\log x)^{5\eps\log d}S}$,  where
$$
S=\sum_{\{m_1,m_2,m_3\}} \frac{1}{[m_1,m_2,m_3]},
$$ 
the summation being over all cliques. The rest of the argument is estimating this sum~$S$.

We write ${S=S'+S''}$, where~$S'$ is the sum over the cliques  with the property
\begin{equation}
\label{eless}
\text{there is a relabeling of the indices such that $[m_1,m_2]< [m_1,m_2,m_3]$,}
\end{equation}
and~$S''$ is over the cliques such that 
\begin{equation}
\label{esame}
[m_1,m_2]=[m_1,m_3]=[m_2,m_3]=[m_1,m_2,m_3].
\end{equation}

\subsection{Estimating~$S'$}
We are starting now to estimate~$S'$. All cliques appearing in this subsection satisfy~\eqref{eless}.

\subsubsection{The estimate with $m_1$ and $m_2$ fixed}
Fix $m_1$ and $m_2$. Then ${m_3\nmid [m_1,m_2]}$ by~\eqref{eless}. Set ${u=\gcd(m_3,[m_1,m_2])}$. With $m_1$ and $m_2$ being fixed, there are at most 
$$
2^{2k}\ll (\log x)^{3\eps\delta}
$$
choices for~$u$ as a divisor of $[m_1,m_2]$.

Writing ${m_3=u v}$. Clearly, ${v\in A}$, where~$A$ is the set from Subsection~\ref{sstwolem}.   Using~\eqref{ecli}, we obtain ${m_1/(2u)\le v\le 2m_1/u}$.
Since~$u$ is a proper divisor of $m_3$, we also have ${v>1}$, which implies ${v\ge y}$, because ${v\in A}$. Also, clearly ${v\le m_3\le x}$. This shows that
\begin{equation}
\label{einter}
\max\left\{y,\frac{m_1}{2u}\right\}\le v\le \min \left\{x,2\frac{m_1}u\right\}.
\end{equation}

We have
${[m_1,m_2,m_3]=[m_1,m_2]v}$. Thus, assuming $m_1$ and $m_2$ fixed, and summing up over all possible~$m_3$, we get %(assuming~$x$ sufficiently large)
\begin{align}
\label{einn}
\sum  \frac{1}{[m_1,m_2,m_3]} &  \le   \frac{1}{[m_1,m_2]} \sum_{u\mid [m_1,m_2]} \  \sum_{v\in A\ \text{satisfying~\eqref{einter}} } \frac{1}{v}\\
& \ll 
\frac{1}{[m_1,m_2] (\log x)^{1-4\eps}} \sum_{u\mid [m_1,m_2]} 1\nonumber\\
& \ll  \frac{1}{(\log x)^{1-8\eps} [m_1,m_2]}. \nonumber
\end{align}
Here, in the inner sum in~\eqref{einn}, we applied Lemma~\ref{lsix} with the choices 
$$
b=\min\left\{x,\frac{2m_1}u\right\}, \quad 
a=\max\left\{y,\frac{m_1}{2u}\right\},
$$ 
and we used the fact that ${\log b-\log a\ll 1}$.

\subsubsection{The estimate with~$m_1$ fixed}
We now fix $m_1$ and vary $m_2$. This time we set ${u=\gcd(m_1,m_2)}$ and again write $m_2=uv$. There are at most ${2^k\ll (\log x)^{2\eps\delta}}$ choices for~$u$. Furthermore, it follows from~\eqref{enot} that~$u$ is a proper divisor of $m_2$, which implies ${v>1}$. Thus, our~$v$ again belongs to the set~$A$ and
satisfies~\eqref{einter}.

Keeping $m_1$ fixed, we argue as above: 
\begin{align*}
\sum \frac{1}{[m_1,m_2]} & =  \frac{1}{m_1} \sum_{u\mid m_1}\ \sum_{v\in A\ \text{satisfying~\eqref{einter}} } \frac1v\\
&\ll  \frac{1}{m_1 (\log x)^{1-4\eps}} \sum_{u\mid m_1} 1\\
& \ll  \frac{1}{m_1 (\log x)^{1-7\eps}}.
\end{align*}

\subsubsection{Estimating~$S'$}
Now we are ready to estimate~$S'$:
$$
S'\ll \frac1{(\log x)^{2-15\eps}} \sum_{m_1\in {\mathcal M}_F'} \frac{1}{m_1}\ll \frac1{(\log x)^{2-17\eps}},
$$
where for the last estimate we used~\eqref{eq:MFprime}.

\subsection{Estimating $S''$}
Now let ${\{m_1,m_2,m_3\}}$ be a clique satisfying~\eqref{esame}. Setting ${u=\gcd(m_1,m_2,m_3)}$ and ${v_i=[m_1,m_2,m_3]/m_i}$, we obtain 
\begin{gather*}
m_1=uv_2v_3,\quad m_2=uv_1v_3, \quad m_3=uv_1v_2, \\ [m_1,m_2]=[m_1,m_3]=[m_2,m_3]=[m_1,m_2,m_3]=uv_1v_2v_3.
\end{gather*}
We again use~\eqref{enot} to obtain ${v_i>1}$, which implies ${v_i\ge y}$ because ${v_i\in A}$. Also,  ${v_i\le x}$.  Together with~\eqref{ecli} this gives
\begin{equation}
\label{enewinter}
\max\left\{y,\frac {v_1}2\right\}\le v_i\le \min\{x, 2v_1\} \qquad (i=2,3). 
\end{equation}

It follows that 
$$
S'' \le \sum_{\substack{u,v_1,v_2,v_3\in A\\
\text{satisfying~\eqref{enewinter}}}}\frac1{uv_1v_2v_3}.
$$
When~$u$ and~$v_1$ are fixed, we have 
$$
\sum_{\substack{v_2,v_3\in A\\
\text{satisfying~\eqref{enewinter}}}}\frac1{uv_1v_2v_3}\le 
\frac1{uv_1}\left(\sum_{\substack{v\in A\\
\max\{y,v_1/2\}\le v\le \min\{x, 2v_1\}}}\frac1v\right)^2,
$$
and the squared sum can be estimated, using Lemma~\ref{lsix}, as ${O\bigl((\log x)^{-1+4\eps}\bigr)}$. Hence 
\begin{equation}
\label{essone}
S''\ll \frac1{(\log x)^{2-8\eps}}\sum\frac1{uv_1},
\end{equation}
the latter sum being over all possible values of~$u$ and~$v_1$. 

To estimate the latter, we make the following observations.

\begin{itemize}
\item
The number $uv_1$ belongs to~$A$, satisfies ${y\le uv_1\le x}$ and ${\omega(yv_1)\le k}$.

\item
Given ${m\in A}$ with ${\omega(m)\le k}$, it can be written as ${m=uv_1}$ in at most ${2^k\ll (\log x)^{2\eps}}$ ways. 
\end{itemize}
It follows that 
\begin{equation}
\label{esstwo}
\sum\frac1{uv_1}\ll (\log x)^{2\eps}\sum_{m\in A\cap[y,x]}\frac1m \ll (\log x)^{6\eps},
\end{equation}
the latter sum being $O\bigl((\log x)^{4\eps}\bigr)$ by Lemma~\ref{lsix} with ${b=x}$ and ${a=y}$. 

Combining~\eqref{essone} and~\eqref{esstwo},  we conclude that 
$$
S'' \ll \frac1{(\log x)^{2-14\eps}}. 
$$

\subsection{Proof of Lemma~\ref{lthree}}

Thus, for large~$x$, the total number of~$n$ such that $F(n)$ has at least $6d$ distinct divisors in $\MM_F(x)$ is bounded by 
$$
x(\log x)^{5\eps\log d}(S'+S'')\ll \frac x{(\log x)^{2-5\eps\log d-17\eps}},
$$
which proves Lemma~\ref{lthree}.

\section{Proof of Theorem~\ref{tmain}}
\label{sptmain}

We are ready now to prove Theorem~\ref{tmain}. Thus, let~$X$ and~$t$ be as in Theorem~\ref{tmain}, and, as in Section~\ref{sargdz}, let  ${F(T)\in \Z[T]}$ be the primitive  separable  polynomial whose roots are exactly the finite critical values of~$t$, with ${d=\deg F}$. We use  all notation and conventions from Section~\ref{sset}. In particular, we fix~$\eps$ satisfying ${0<\eps\le 1/2}$ (which will be specified later)  and for sufficiently large~$x$ we consider the set $\MM_F(x)$. 

Recall (see Section~\ref{sargdz}) that we denote by~$K_n$ the field $\Q(t^{-1}(n))$.  We call ${m\in \MM_F}$ primitive for~$n$ if every ${p\mid m}$ ramifies in~$K_n$, but for every ${n'<n}$ some ${p\mid m}$ does not ramify in~$K_{n'}$. Clearly, if~$n$ admits a primitive ${m\in \MM_F}$ then the field~$K_n$ is distinct from ${K_1, \ldots, K_{n-1}}$. 

Our starting point is Corollary~\ref{cprim}, which asserts that every ${m\in \MM_F}$ with the property ${\pmin(m)>\omega(m)}$  serves as a primitive for some ${n_m\le m( \omega(m)+1)}$. If ${m\in \MM_F(x)}$ then this property is trivially satisfied when~$x$ is large enough; hence every  ${m\in \MM_F(x)}$ serves as primitive for some ${n_m\le m(k+2)}$, and we have
\begin{equation}
\label{elex}
n_m\le m(k+2)\le \frac x{\log\log x}{(\eps\delta\log\log x+3)}\le x,
\end{equation}
again provided~$x$ is sufficiently large.

 Set 
\begin{align*}
\NN(x)&=\{n_m: m\in \MM_F(x)\},\\ 
\NN'(x)&=\{n\in \NN(x): \text{the fiber $t^{-1}(n)$ is $\Q$-irreducible}\}. 
\end{align*}
It follows from~\eqref{elex} that
$$
\NN'(x)\subset\NN(x)\subset [1,x],
$$
and Hilbert's Irreducibility Theorem implies that 
\begin{equation}
\label{ehit}
|\NN'(x)|\ge |\NN(x)|-O(x^{1/2}). 
\end{equation}
The fields 
$$
K_n \qquad (n\in \NN(x)\ )
$$
are pairwise distinct, and, since for ${n\in \NN'(x)}$ the field~$K_n$ is the Galois closure of $\Q(P_n)$, the fields 
\begin{equation}
\Q(P_n)\qquad (n\in \NN'(x)\ )
\end{equation}
are pairwise distinct as well. 

Thus, to prove Theorem~\ref{tmain}, we only have to show that, with suitable choice of~$\eps$,  the lower estimate 
\begin{equation}
\label{etired}
|\NN'(x)|\ge \frac x{(\log x)^{1-\eta}}
\end{equation}
holds for sufficiently large~$x$. Here~$\eta$ is a positive number depending only on~$d$ (which, through~\eqref{erihu}, translates into dependence in~$\nu$ and~$\bfg$). 

This can be accomplished using the results of Sections~\ref{sset} and~\ref{sgrege}. Since every ${p\mid m}$ ramifies in~$K_{n_m}$, we have ${m\mid F(n_m)}$ (see Property~\ref{iramdiv} in Section~\ref{sargdz}). Hence Proposition~\ref{psect} applies to our definition of~$n_m$. Thus, setting~$\eps$ as in~\eqref{eeps},  Proposition~\ref{psect} implies that, for sufficiently large~$x$, we have 
${|\NN(x)|\ge (12d)^{-1}|\MM_F(x)|}$. Together with Proposition~\ref{passy} this implies that
${|\NN(x)|\ge x(\log x)^{-1+\delta\eps+o(1)}}$ as ${x\to \infty}$, which, combined with~\eqref{ehit}, implies the same lower estimate for $|\NN'(x)|$. In particular, for sufficiently large~$x$ we have~\eqref{etired} with ${\eta=\delta\eps/2}$. 

In view of~\eqref{edelffac} and~\eqref{eeps} we have 
${\eta\ge 10^{-4}(d\log(2d))^{-1}}$. 
Using~\eqref{erihu} we deduce that ${\eta\ge 10^{-6}\bigl((\bfg+\nu)\log(\bfg+\nu)\bigr)^{-1}}$. \qed

%\renewcommand\thesection{A}

%%%%%%%%%%%%%%%%%%%%%%%%%%%%%%%%%%%%%%%%%%%%%

%%%%%%%%%%%%%%%%%%%%%%%%%%%%%%%%%%%%%%%%%%%%%

\end{document}